\documentclass[secthm,seceqn,amsthm,ussrhead,reqno, 12pt]{amsart}
\usepackage[utf8]{inputenc}
\usepackage[english]{babel}
\usepackage[symbol]{footmisc}
\usepackage{amssymb,amsmath,amsthm,amsfonts,xcolor,enumerate,hyperref,comment,longtable,cleveref}

\usepackage{times}
\usepackage{cite}
\usepackage{pdflscape}
\usepackage{ulem}
 
\usepackage{tikz}
\usepackage{hyperref}
\usepackage{cancel}
\usepackage{stmaryrd}

\newtheorem{theorem}{Theorem}

\newtheorem{lemma}[theorem]{Lemma}

\newtheorem{definition}[theorem]{Definition}

\crefrangeformat{equation}{#3(#1)#4--#5(#2)#6}

\crefname{theorem}{theorem}{Theorems}
\crefname{lem}{Lemma}{Lemmas}
\crefname{cor}{Corollary}{Corollaries}
\crefname{prop}{Proposition}{Propositions}
\crefname{proposition}{Proposition}{Propositions}
\crefname{defn}{Definition}{Definitions}
\crefname{exm}{Example}{Examples}
\crefname{rem}{Remark}{Remarks}
\crefname{section}{Section}{Sections}
\crefname{equation}{\unskip}{\unskip}
\crefname{enumi}{\unskip}{\unskip}

\newcommand{\gen}[1]{\langle #1\rangle}

\newcommand{\red}[1]{\textcolor{black}{#1}}

\usepackage{fouriernc}

\begin{document}

\noindent{\Large 
Transposed Poisson structures on Virasoro-type algebras}
 \footnote{
The  first part of the work is supported by 
FCT   UIDB/00212/2020, UIDP/00212/2020, and 2022.02474.PTDC.
The second part of this work is supported by the Russian Science Foundation under grant 22-71-10001.
} 

	\bigskip
	
	 \bigskip

\begin{center}	
	{\bf
		Ivan Kaygorodov\footnote{CMA-UBI, Universidade da Beira Interior, Covilh\~{a}, Portugal; 
   \ Saint Petersburg  University, Russia;\  kaygorodov.ivan@gmail.com},
   Abror Khudoyberdiyev\footnote{V.I.Romanovskiy Institute of Mathematics Academy of Science of Uzbekistan; National University of Uzbekistan; \ 
khabror@mail.ru} \&
Zarina Shermatova\footnote{V.I.Romanovskiy Institute of Mathematics Academy of Science of Uzbekistan; Kimyo International University in Tashkent, Uzbekistan; \ 
z.shermatova@mathinst.uz}}   
\end{center}

 \bigskip
  
\noindent {\bf Abstract.}
{\it 
We compute  $\frac{1}{2}$-derivations on the deformed generalized Heisenberg-Virasoro\footnote{We refer the notion of \textquotedblleft Witt\textquotedblright  \   algebra to the simple Witt algebra and the notion of \textquotedblleft Virasoro\textquotedblright  \  algebra to the central extension of the simple Witt algebra.} algebras and  
on not-finitely graded Heisenberg-Virasoro algebras   
$\widehat{W}_n(G)$, $\widetilde{W}_n(G)$,   and $\widetilde{HW}_n(G)$.
We classify all transposed Poisson structures on
such algebras. }



 \bigskip
\section*{Introduction} 

Since their origin in the 1970s in Poisson geometry, Poisson algebras have appeared in several areas of mathematics and physics, such as algebraic geometry, operads, quantization theory, quantum groups, and classical and quantum mechanics. One of the natural tasks in the theory of Poisson algebras is the description of all such algebras with a fixed Lie or associative part.
	Recently, Bai, Bai, Guo,  and Wu have introduced a dual notion of the Poisson algebra~\cite{bai20}, called a transposed Poisson algebra, by exchanging the roles of the two multiplications in the Leibniz rule defining a Poisson algebra. A transposed Poisson algebra defined this way not only shares some properties of a Poisson algebra, such as the closedness under tensor products and the Koszul self-duality as an operad but also admits a rich class of identities \cite{kms,bai20,fer23,lb23}. It is important to note that a transposed Poisson algebra naturally arises from a Novikov-Poisson algebra by taking the commutator Lie algebra of its Novikov part \cite{bai20}.
 	Any unital transposed Poisson algebra is
	a particular case of a ``contact bracket'' algebra 
	and a quasi-Poisson algebra.
 Each transposed Poisson algebra is a 
 commutative  Gelfand-Dorfman algebra \cite{kms}
 and it is also an algebra of Jordan brackets \cite{fer23}.
	In a recent paper by Ferreira, Kaygorodov, and Lopatkin
	a relation between $\frac{1}{2}$-derivations of Lie algebras and 
	transposed Poisson algebras has been established \cite{FKL}. 	These ideas were used to describe all transposed Poisson structures 
	on  Witt and Virasoro algebras in  \cite{FKL};
	on   twisted Heisenberg-Virasoro,   Schr\"odinger-Virasoro  and  
	extended Schr\"odinger-Virasoro algebras in \cite{yh21};
	on Schr\"odinger algebra in $(n+1)$-dimensional space-time in \cite{ytk};
\red{on solvable Lie algebra with filiform nilradical} in \cite{aae23};
on oscillator Lie algebras in \cite{kkh24};
on Galilean Lie algebras in \cite{klv22};
	on Witt type Lie algebras in \cite{kk23};
	on generalized Witt algebras in \cite{kkg23}; 
 Block Lie algebras in \cite{kkg23}
  and
on    Lie incidence algebras (for all references, see the survey \cite{k23}).
  
 \medskip

The description of transposed Poisson structures obtained on the Witt algebra \cite{FKL} opens the question of finding algebras related to the Witt algebra
which admit nontrivial transposed Poisson structures. 
So, some algebras related to Witt algebra are studied in \cite{kk23,kkg23}.
The present paper is a continuation of this research. 
Specifically, we describe transposed Poisson structures on 
the deformed generalized Heisenberg-Virasoro algebra $g(G, \lambda)$
and not-finitely graded Lie algebras  $\widehat{W}(G),$  $\widetilde{W}(G),$ and $\widetilde{HW}(G).$
The notion of deformed generalized Heisenberg-Virasoro (${\rm dgHV}$) algebra $g(G, \lambda)$ 
has been appeared in \cite{x0}.
It appears as the central extension of the deformed higher-rank Heisenberg-Virasoro algebra, 
which is a generalization of algebras ${\mathcal W}(0,b)$, defined in \cite{kac}. 
Verma modules over ${\rm dgHV}$ algebras were studied in  \cite{x21}.
All transposed Poisson structures on  ${\mathcal W}(0,b)$ were described in \cite{FKL}.
The algebra $g(\mathbb Z, 0)$ is named twisted Heisenberg-Virasoro algebra and it was well-studied \cite{ackp}.
So, 
there are works about 
derivations and automorphisms \cite{sj06}, 
representations \cite{b03,LZ20}, 
modules \cite{LZ,chsy},
Lie bialgebra structures \cite{lp12}, 
left-symmetric algebra structures \cite{cl14},
transposed Poisson structures  \cite{yh21},
and so on.
The algebra $\widehat{W}(G)$ was introduced in \cite{chs15} as the central extension of the algebra  ${W}(G).$
Lie bialgebra structures on $\widehat{W}(G)$ were studied in \cite{chs14}.
All transposed Poisson structures on $W(G)$ were described in \cite{kkz}.
Algebras $\widetilde{W}(G)$ and $\widetilde{HW}(G),$ as the central extensions of suitable Lie algebras, were introduced in \cite{su} and \cite{fzy15}, respectably.
All transposed Poisson structures on $HW(G)$ were described in \cite{kkz}.

 \medskip

Summarizing, we prove that algebras $g(G, \lambda)$  admit nontrivial transposed Poisson structures only for $\lambda=-1$;
algebras $\widehat{W}(G)$ admit nontrivial transposed Poisson structures; 
algebras $\widetilde{W}(G)$ and $\widetilde{HW}(G)$ do not admit nontrivial transposed Poisson structures.

\section{Preliminaries}

In this section, we recall some definitions and known results for studying transposed Poisson structures. Although all algebras and vector spaces are considered over the complex field, many results can be proven over other fields without modifications of proofs. The notation $\gen{S}$ means the $\mathbb C$-subspace generated by $S$.

\begin{definition} 
Let $\mathfrak {L}$ be a vector space equipped with two nonzero bilinear operations $\cdot$ and $[\cdot, \cdot]$. The triple $(\mathfrak {L}, \cdot, [\cdot, \cdot])$ is called a transposed Poisson algebra if $(\mathfrak {L}, \cdot)$ is a commutative associative  algebra and $(\mathfrak {L}, [\cdot, \cdot])$ is a Lie algebra that satisfies the following compatibility condition
$$
\begin{array}{c}
2z \cdot [x,y]=[z \cdot x, y]+[x, z \cdot y].
\end{array}
$$
\end{definition}

\begin{definition} 
Let $(\mathfrak{L}, [\cdot, \cdot])$ be a Lie algebra. A transposed Poisson structure on $(\mathfrak {L}, [\cdot, \cdot])$ is a commutative associative  multiplication $\cdot$ in $\mathfrak {L}$ which makes $(\mathfrak {L}, \cdot, [\cdot, \cdot])$ a transposed Poisson algebra.
\end{definition}

\begin{definition} 
Let $(\mathfrak {L}, [\cdot, \cdot])$ be a Lie algebra, $\varphi: \mathfrak {L}\rightarrow \mathfrak {L} $ be a linear map. Then $\varphi$ is a $\frac 12$-derivation if it satisfies
$$
\begin{array}{c}
\varphi ([x,y])= \frac 12 \big([\varphi (x), y]+[x, \varphi (y)]\big).
\end{array}
$$
\end{definition}

Observe that $\frac 12$-derivations are a particular case of $\delta$-derivations introduced by Filippov in 1998 (for references about the  study of $\frac{1}{2}$-derivations, see \cite{k23}) and 
recently the notion of $\frac{1}{2}$-derivations of algebras was generalized to 
 $\frac{1}{2}$-derivations from an algebra to a module \cite{zz}.  
The main example of $\frac 12$-derivations is the multiplication by an element from the ground field. Let us call such $\frac 12$-derivations as  {trivial
$\frac 12$-derivations.} It is easy to see that $[\mathfrak {L},\mathfrak {L}]$ and $\operatorname{Ann}(\mathfrak {L})$ are invariant under any $\frac 12$-derivation of $\mathfrak {L}$.

Let $G$ be an abelian group, $\mathfrak{L}=\bigoplus \limits_{g\in G}\mathfrak{L}_{g}$ be a $G$-graded Lie algebra.
We say that a $\frac 12$-derivation $\varphi$ has degree $g$ (and denoted by $\deg(\varphi)$) if $\varphi(\mathfrak{L}_{h})\subseteq \mathfrak{L}_{g+h}$.
Let $\Delta(\mathfrak{L})$ denote the space of $\frac 12$-derivations and write $\Delta_{g}(\mathfrak{L})=\{\varphi \in \Delta(\mathfrak{L}) \mid \deg(\varphi)=g\}$.
The following trivial lemmas are useful in our work.

\begin{lemma}\label{l01}
Let $\mathfrak{L}=\bigoplus \limits_{g\in G}\mathfrak{L}_{g}$ be a $G$-graded Lie algebra. Then
$
\Delta(\mathfrak{L})=\bigoplus \limits_{g\in G}\Delta_{g}(\mathfrak{L}).
$
\end{lemma}

\begin{lemma}\label{l1} (see {\rm\cite{FKL}})
Let $(\mathfrak {L}, \cdot, [\cdot, \cdot])$ be a transposed Poisson algebra and $z$ an arbitrary element from $\mathfrak {L}$.
Then the left multiplication $L_{z}$ in the   commutative associative  algebra $(\mathfrak {L}, \cdot)$ gives a $\frac 12$-derivation on the Lie algebra $(\mathfrak {L}, [\cdot, \cdot])$.
\end{lemma}

\begin{lemma}\label{lempr}  (see {\rm\cite{FKL}})
Let $\mathfrak {L}$ be a Lie algebra without non-trivial $\frac 12$-derivations. Then every transposed Poisson structure defined on $\mathfrak {L}$ is trivial.
\end{lemma}

\section{Transposed Poisson structures on deformed generalized Heisenberg-Virasoro algebras}

Let $\lambda \in \mathbb C$ and  $G$ be an additive subgroup of $\mathbb{C}$ such that $G$ is free of rank $\nu \geq 1$ if $\lambda= -2.$ 
The deformed generalized Heisenberg-Virasoro algebra $g(G, \lambda)$  was obtained by the universal central extension of the algebra $\check{g}(G,\lambda)$ with multiplications
$$[L_{a}, L_{b}] = (b - a)L_{a+b}, \quad [L_{a}, I_{b}] = (b -\lambda a)I_{a+b}.$$
Recall from
\cite{x0} the deformed generalized Heisenberg-Virasoro algebra $g(G, \lambda)$\footnote{
Let us mention, that  the algebra  $g(G, \lambda)$ from \cite{x0} was considered under the following restrictions 
 $\mathbb Z \subseteq G$ and $\frac{1}{k} \notin G$ for any integer $k > 1$.
The algebra $\check{g}(G,-1)$ has no universal central extension, hence  {\rm dgHV}
 algebra $g(G, \lambda)$ usually considered with $\lambda\neq -1.$ 
 In the present  work, we consider the algebra $g(G, \lambda)$ without restrictions on $G$ and $\lambda$.} with one deformation parameter $\lambda$
is a Lie algebra with a basis 
$\{L_{a}, I_{a}, C_{L},  \ | \ a \in G\}_{\lambda \neq 0,1,-2}$ 
$\big($resp., $\{L_{a}, I_{a}, C_{L}, C_{I}, C_{LI}^{(0)} \ | \ a \in G\}_{\lambda = 0},$  \ 
$\{L_{a}, I_{a}, C_{L},  C_{LI}^{(1)} \ | \ a \in G \}_{\lambda = 1},$  \ 
$\{L_{a}, I_{a}, C_{L},  C_{LI}^{(i)} \ | \ a \in G, 2 \leq i \leq \nu\}_{\lambda = -2}\big)$ 
and the multiplication table is given by
\begin{longtable}{rcl}
$[L_{a},L_{b}]$&$=$&$(b-a)L_{a+b}+\frac{1}{12}(a^{3}-a)\delta_{a,-b}C_{L}$ \\ 
$[L_{a},I_{b}]$&$=$&$(b-\lambda a)I_{a+b}+\delta_{a,-b}\left ((a^{2}+a)\delta_{\lambda,0}C_{LI}^{(0)}+\frac{1}{12}(a^{3}-a)\delta_{\lambda,1}C_{LI}^{(1)}+\sum\limits_{i=2}^{\nu}a_{(i)}\delta_{\lambda,-2}C_{LI}^{(i)}\right)$\\
$[I_{a},I_{b}]$&$=$&$aC_{I}\delta_{a,-b}\delta_{\lambda,0},$   
\end{longtable}
\noindent where $a=\sum\limits_{i=1}^{\nu}a_{(i)}\epsilon_{1}$ and  $a_{(i)} \ (1 \leq i \leq \nu)$ are coefficients of a with respect to
a fixed $\mathbb{Z}$-basis of $\epsilon_{1},\ldots, \epsilon_{\nu}$ of $G$ for the case $\lambda = -2$.

In \cite{yh21} the authors showed that every $\frac{1}{2}$-derivation of $g(\mathbb{Z},0)$ is trivial. 
It should be noted, that the generalization of the present result to the case of 
 $g(G,0)$ is similar of the proof given in \cite{yh21} and we will omitted it.
So we will consider the case of $\lambda\neq 0,$ then our multiplication table   becomes the following relations

\begin{eqnarray}
\label{ph1}
[L_{a},L_{b}]&=&(b-a)L_{a+b}+\frac{1}{12}(a^{3}-a)\delta_{a,-b}C_{L}, \\[1mm]
\label{ph2}
[L_{a},I_{b}]&=&(b-\lambda a)I_{a+b}+\delta_{a,-b}\left (\frac{1}{12}(a^{3}-a)\delta_{\lambda,1}C_{LI}^{(1)}+\sum_{i=2}^{\nu}a_{(i)} \delta_{\lambda,-2}C_{LI}^{(i)}\right).
\end{eqnarray}

Set
    $g(G,\lambda)_0=\langle L_0, I_0, C_L, C_{LI}^{(i)} \ | \ 1\leq i\leq \nu \rangle$ and $g(G,\lambda)_{j\neq 0}=\langle L_j,I_j\rangle.$
Then $g(G,\lambda)=\bigoplus\limits_{j\in G}g(G,\lambda)_j$ and 
$\Delta(g(G,\lambda))=\bigoplus\limits_{j\in G}\Delta_{j}(g(G,\lambda)).$
To determine $\Delta(g(G,\lambda)),$ we need to calculate 
every homogeneous subspace $\Delta_{j}(g(G,\lambda)).$
Let $\varphi_{j} \in \Delta_{j}(g(G,\lambda)).$ 
Since $\operatorname{Ann}({\mathfrak{L}})$ is invariant under any $\frac{1}{2}$-derivation we can assume
\begin{longtable}{lll}
 $\varphi_{j}(L_{k})$ & = &$\alpha_{j,k}L_{k+j}+\alpha'_{j,k}I_{k+j}+\delta_{k+j,0}\left(\gamma_{j}\delta_{\lambda,1}C_{L}+\delta_{\lambda,-2}\sum\limits_{i=1}^{\nu}\gamma^{(i)}_{j}C_{LI}^{(i)}\right) $\\
$\varphi_{j}(I_{k})$ & = &$\beta_{j,k}L_{k+j}+\beta'_{j,k}I_{k+j}+\delta_{k+j,0}\left(\mu_{j}\delta_{\lambda,1}C_{L}+\delta_{\lambda,-2}\sum\limits_{i=1}^{\nu}\mu^{(i)}_{j}C_{LI}^{(i)}\right)$\\
$\varphi_{j}(C_{L})$ & = &$\delta_{j,0}\left(d_0C_{L}+\delta_{\lambda,1}d_0^{(1)}C_{LI}^{(1)}+\delta_{\lambda,-2}\sum\limits_{i=2}^{\nu}d_0^{(i)}C_{LI}^{(i)}\right)$\\
$\varphi_{j}(C^{(k)}_{LI})$ & = &$\delta_{j,0}\left(d_{k}C_{L}+\delta_{\lambda,1}d^{(1)}_{k}C_{LI}^{(i)}+\delta_{\lambda,-2}\sum\limits_{i=2}^{\nu}d^{(i)}_{k}C_{LI}^{(i)}\right)$
\end{longtable}

Applying $\varphi_{j}$ to both sides of (\ref{ph1}), we have 

{\footnotesize{\begin{eqnarray}&&2(b-a)\alpha_{j,a+b}= \alpha_{j,a}(b-a-j)+ \alpha_{j,b}(b-a+j),\label{32}\end{eqnarray}
\begin{eqnarray}&&2(b-a)\alpha'_{j,a+b}=-\alpha'_{j,a}(a-\lambda b+j)+\alpha'_{j,b}(b-\lambda a+j),\label{33}\end{eqnarray}
\begin{eqnarray}&&\delta_{a+b,-j}(b-a)\gamma_{j}+\frac{1}{12}\delta_{j,0}\delta_{a,-b}(a^{3}-a)d_0=\frac{1}{24}\delta_{a+b,-j}\left(((a+j)^{3}-(a+j))\alpha_{j,a}+(a^3-a)\alpha_{j,b}\right),\label{34}\end{eqnarray}
\begin{eqnarray}&&\delta_{a+b,-j}(b-a)\gamma^{(1)}_{j}+\frac{1}{12}\delta_{j,0}\delta_{a,-b}(a^{3}-a)d_0^{(1)}=\frac{1}{24}\delta_{a+b,-j}((a^3-a)\alpha'_{j,b}-(b^3-b)\alpha'_{j,a}), \mbox{ if } \lambda=1,\label{c.1}\end{eqnarray}
\begin{eqnarray}&&\delta_{a+b,-j}(b-a)\gamma^{(i)}_{j}+\frac{1}{12}\delta_{j,0}\delta_{a,-b}(a^{3}-a)d_0^{(i)}=\frac{1}{2}\delta_{a+b,-j}(a_{(i)}\alpha'_{j,b}-b_{(i)}\alpha'_{j,a}), \ \ 2\leq i\leq\nu,  \mbox{ if } \lambda=-2.\label{35}\end{eqnarray}}}

Applying $\varphi_{j}$ to both sides of  (\ref{ph2}), we have
{\footnotesize{\begin{eqnarray}
&&2(b-\lambda a)\beta_{j,a+b}= \beta_{j,b}(b-a+j),\label{36}
\end{eqnarray}
\begin{eqnarray}&&2(b-\lambda a)\beta'_{j,a+b}= \alpha_{j,a}(b-\lambda a-\lambda j)+ \beta'_{j,b}(b-\lambda a+j),\label{37}
\end{eqnarray}
\begin{eqnarray}
&&\delta_{a+b,-j}(b-\lambda a)\mu_{j}+\delta_{j,0}\delta_{a,-b}\left(\frac{1}{12}\delta_{\lambda,1}(a^3-a)d_{1}+\delta_{\lambda,-2}\sum_{i=2}^{\nu}a_{(i)}d_{i}\right)=\frac{1}{24}\delta_{a+b,-j}(a^{3}-a)\beta_{j,b},\label{38}
\end{eqnarray}
\begin{eqnarray}
&&\delta_{a+b,-j}(b-\lambda a)\mu^{(1)}_{j}+\delta_{j,0}\delta_{a,-b}\frac{1}{12}(a^3-a)d_{1}^{(1)}=\frac{1}{24}\delta_{a+b,-j}\left(((a+j)^3-(a+j))\alpha_{j,a}+(a^3-a)\beta'_{j,b}\right),  \mbox{ if } \lambda=1, \label{c.2}
\end{eqnarray}
\begin{eqnarray}
&&\delta_{a+b,-j}(b-\lambda a)\mu^{(i)}_{j}+\delta_{j,0}\delta_{a,-b}\sum_{t=2}^{\nu}a_{(t)}d_{t}^{(i)}=
\frac{1}{2}\delta_{a+b,-j}\left((a+j)_{(i)}\alpha_{j,a}+a_{(i)}\beta'_{j,b}\right),2\leq i\leq\nu,  \mbox{ if } \lambda=-2.\label{39}
\end{eqnarray}}}

\begin{lemma}\label{MainLemma} 
$ \Delta_{j\neq0}(g(G,\lambda))=0.$ 
\end{lemma}

\begin{proof} Let $j\neq 0$ and  $\varphi_{j} \in \Delta_{j}(g(G,\lambda)),$ then we consider the following cases.

\begin{enumerate}
    \item[$\bullet$] 
First, we consider the relations (\ref{32}), (\ref{34}) and (\ref{37}).
\begin{enumerate}[(a)]
    \item 
Putting $b=0$ in (\ref{32}), we have 
$(j-a)\alpha_{j,a}=(j-a)\alpha_{j,0},$ which implies 
$\alpha_{j,a}=\alpha_{j,0}$ for $a\neq j.$ 
    \item Setting $b=-a=j$ in (\ref{32}), we get $\alpha_{j,0}=\alpha_{j,j},$ which derives
$\alpha_{j,a}=\alpha_{j,0}$ for all $a \in G.$ 
\item If  $\lambda \neq1,$ then  taking $a=0,$ $b=j$ in (\ref{37}), we have $j(\lambda -1)\alpha_{j,0}=0,$ implies $\alpha_{j,0}=0.$  
    \item 
If  $\lambda =1,$ then taking $b=0$ and $a=-j$ in (\ref{34}), we obtain  $\gamma_{j}=\frac{1}{24}\alpha_{j,0}(1-j^{2}).$ 
 Setting $b=-a-j$ in (\ref{34}), we get $(j+a)(aj+2a^2+1)\alpha_{j,0}=0.$  
 Choosing $a$ such that $(j+a)(aj+2a^2+1)\neq0,$ we obtain
$\alpha_{j,0}=0.$ 

\item[$\star$] Hence, 
$\alpha_{j,a}=0$ for all $a\in G.$

\end{enumerate}

  \item[$\bullet$] Second, we  consider the relations (\ref{33}) and (\ref{c.1}). 
\begin{enumerate}[(a)]
    \item If  $\lambda \neq 1,$ then setting $b=0,$ $a=j$ in (\ref{33}), we have $j(\lambda - 1) \alpha'_{j,0}=0,$ which implies $\alpha'_{j,0}=0.$
Taking $b=0,$ $a\neq j$ in (\ref{33}), we get $\alpha'_{j,a}=0$ for $a\neq j.$ Finally, setting $a=j,$ $b=2j$ in (\ref{33}) gives $j(\lambda - 1) \alpha'_{j,j}=0,$ which derives $\alpha'_{j,j}=0.$

\item If $\lambda = 1,$ then similarly to the previous case using (\ref{33}) with $b=0,$  we have 
$(j-a)\alpha_{j,a}'=(j-a)\alpha_{j,0}',$ which implies 
$\alpha_{j,a}'=\alpha_{j,0}'$ for $a\neq j.$ Then setting $b=-a=j$ in (\ref{33}), we get $\alpha_{j,0}'=\alpha_{j,j}',$ which derives
$\alpha_{j,a}'=\alpha_{j,0}'$ for all $a \in G.$ 
Now taking $b=0$ and $a=-j$ in (\ref{c.1}), we obtain  $\gamma_{j}^{(1)}=\frac{1}{24}\alpha_{j,0}'(1-j^{2}).$ Then
setting $b=-a-j$ in (\ref{c.1}), we get $a(j+a)(j+2a)\alpha_{j,0}'=0.$  Choosing $a\notin \{-\frac{j}{2};-j;0\},$ we obtain
$\alpha_{j,0}'=0.$ 
\item[$\star$]  
Hence, $\alpha_{j,a}'=0$ for all $a\in G.$

\end{enumerate}


\item[$\bullet$] Third, we consider the relation (\ref{36}).
\begin{enumerate}[(a)]
    \item
Setting $a=0$ in (\ref{36}), we derive $(b-j)\beta_{j,b}=0,$ which implies $\beta_{j,b}=0$ for $b\neq j.$
\item Letting $b=-a=j$ in (\ref{36}) gives $\beta_{j,j}=0.$
\item[$\star$]  
Hence, $\beta_{j,b}=0$ for all $b\in G.$
\end{enumerate}

\item[$\bullet$] Fourth, we consider the relation  (\ref{37}).
\begin{enumerate}[(a)]

\item Taking $a=0$ in (\ref{37}), we have $(b-j)\beta'_{j,b}=0,$ which derives $ \beta'_{j,b}=0$ for $b\neq j.$

\item Using (\ref{37}) with $a=-j,$ $b=j$ and $a=j,$ $b=0$
we obtain $j(\lambda +2)\beta'_{j,j}=0$ 
and $j\lambda \beta'_{j,j}=0,$ which implies $\beta'_{j,j}=0.$ 

\item[$\star$]  
Hence,  $\beta'_{j,b}=0$ for all $b\in G.$
\end{enumerate}
\end{enumerate}

Finally, using (\ref{34}), (\ref{c.1}), (\ref{35}), (\ref{38}), (\ref{c.2}) and (\ref{39}) with $b=0,$ $a=-j,$ we obtain \begin{center}
    $\gamma_{j}=0,$ $\gamma^{(i)}_{j}=0,$ $\mu_{j}=0$
and $\mu^{(i)}_{j}=0$ for $1\leq i\leq\nu.$ 
\end{center}
This proves that $\varphi_{j}=0$ for all $j\neq 0.$
\end{proof}

\begin{lemma}
 \label{thm9} 
If $\lambda\notin \{-2; -1; 1\},$ then 
$\Delta(g(G,\lambda))=\langle {\rm Id} \rangle.$ \end{lemma}

\begin{proof}
Thanks to Lemma \ref{MainLemma}, it is sufficient to consider only $\frac{1}{2}$-derivation  $\varphi_{0}$, i.e., $j=0.$
Let us consider the following steps.
\begin{enumerate}[(a)]
    \item Letting $b=0$ in (\ref{32}), we have $a\alpha_{0,a}=a\alpha_{0,0},$ which implies $\alpha_{0,a}=\alpha_{0,0}$  for all $a\in G.$

\item Taking $b=0$ in (\ref{33}), we get
$\alpha'_{0,a}=\lambda\alpha'_{0,0}.$ Then, putting $a\neq 0,$ and  $b=2a$ in (\ref{33}), we obtain $\lambda(\lambda-1)\alpha'_{0,0}=0.$ 
Hence, $\alpha'_{0,a}=0$ for all $a\in G.$

\item Now, setting $a=0$ in (\ref{36}), we derive $b\beta_{0,b}=0,$ which implies $\beta_{0,b}=0$ for $b\neq 0.$
Letting $b=-a$ in (\ref{36}), we obtain $\beta_{0,0}=0.$
Hence, $\beta_{0,b}=0$ for all $b\in G.$

\item Putting $a=0$ in (\ref{37}) gives $b\beta'_{0,b}=b\alpha_{0,0}$, which derives $\beta'_{0,b}=\alpha_{0,0}$ for $b\neq 0.$
Taking $a\neq 0,$ $b=-a,$ in (\ref{37}), we get $(\lambda+1)(\beta'_{0,0}-\alpha_{0,0})=0.$
Hence, $\beta'_{0,b}=\alpha_{0,0}$ for all $b\in G.$

\item Using (\ref{34}) with $a=1,$ $b=-1,$ we obtain $\gamma_{0}=0$ and 
with $a^3\neq a,$ $b=-a,$ we get $d_0=\alpha_{0,0}.$

\item Now putting $b=-a$ in (\ref{38}), we have $(\lambda+1)\mu_{0}=0.$
Hence, we have $\mu_{0}=0.$ 

\end{enumerate}
As a conclusion, we obtain $\varphi_0=\alpha_{0,0}\textrm{Id}.$
\end{proof}

\begin{lemma}\label{thm(-2)}
$\Delta(g(G,-2)=\langle {\rm Id} \rangle.$    
\end{lemma}

\begin{proof}
Thanks to Lemma \ref{MainLemma}, it is sufficient to consider only $\frac{1}{2}$-derivation  $\varphi_{0}$.
In the case of $\lambda =-2$ and $j=0$ the restrictions \eqref{32} -- \eqref{39} have the following form:
{\footnotesize{\begin{eqnarray} 
&&2(b-a)\alpha_{0,a+b}=(b-a)(\alpha_{0,a}+\alpha_{0,b}),\label{R1}
\end{eqnarray}
\begin{eqnarray}
&&2(b-a)\alpha'_{0,a+b}=\alpha'_{0,b}(b+2a)-\alpha'_{0,a}(2b+a),\label{R2}
\end{eqnarray}
\begin{eqnarray}
&&2(a^{3}-a)d_{0}-48a\gamma_{0}=(a^{3}-a)(\alpha_{0,a}+\alpha_{0,-a}),\label{R3}\end{eqnarray}
\begin{eqnarray}
&&(a^{3}-a)d^{(i)}_{0}-24a\gamma^{(i)}_{0}=
6a_{(i)}(\alpha'_{0,-a}+\alpha'_{0,a}), \quad 2\leq i\leq\nu,\label{R4}\end{eqnarray}
\begin{eqnarray}
&&2(b+2a)\beta_{0,a+b}=\beta_{0,b}(b-a),\label{R5}\end{eqnarray}
\begin{eqnarray}
&&2(b+2a)\beta'_{0,a+b}=(b+2a)(\alpha_{0,a}+\beta'_{0,b}),\label{R6}\end{eqnarray}
\begin{eqnarray}
&&24a\mu_{0}+24\sum_{i=2}^{\nu}a_{(i)}d_{i}=(a^{3}-a)\beta_{0,-a}, \quad 2\leq i\leq\nu,\label{R7}\end{eqnarray}
\begin{eqnarray}
&&2a\mu^{(i)}_{0}+2\sum_{t=2}^{\nu}a_{(t)}d_{t}^{(i)}=(\alpha_{0,a}+\beta'_{0,-a})a_{(i)}, \quad 2\leq i\leq\nu. \label{R8}
\end{eqnarray}}}

\begin{enumerate}

\item[$\bullet$] First, we  consider the relations (\ref{R1}) and (\ref{R2}). 

\begin{enumerate}[(a)]
    \item 
Letting $b=0$ in (\ref{R1}), we have $a\alpha_{0,a}=a\alpha_{0,0},$ which implies $\alpha_{0,a}=\alpha_{0,0}$  for all $a\in G.$

\item Taking $b=0$ in (\ref{R2}), we get
$\alpha'_{0,a}=-2\alpha'_{0,0}.$ 

\item Then, putting $a\neq 0,$ $b=2a$ in (\ref{R2}), we obtain $\alpha'_{0,0}=0.$ Hence, $\alpha'_{0,a}=0$ for all $a\in G.$
\end{enumerate}

\item[$\bullet$] Second, we  consider   the relations (\ref{R5}) and (\ref{R6}). 
\begin{enumerate}[(a)]
\item Now, setting $a=0$ in (\ref{R5}), we derive $b\beta_{0,b}=0,$ which implies $\beta_{0,b}=0$ for $b\neq 0.$

\item Letting $b=-a$ in (\ref{R5}), we obtain $\beta_{0,0}=0.$
Hence $\beta_{0,b}=0$ for all $b\in G.$

\item Putting $a=0$ in (\ref{R6}) gives $b\beta'_{0,b}=b\alpha_{0,0}$, which derives $\beta'_{0,b}=\alpha_{0,0}$ for $b\neq 0.$

\item Taking $b=-a$ in (\ref{R6}), we get $\beta'_{0,0}=\alpha_{0,0}.$
\end{enumerate}

\item[$\bullet$] Third, we consider the relations (\ref{R3}) and (\ref{R4}). We have the following cases.

\begin{enumerate}[(a)]
    \item Using (\ref{R3}) with $a=1,$ we obtain $\gamma_{0}=0$ and 
with $a^3\neq a,$ we get $d_0=\alpha_{0,0}.$

\item Similarly, from (\ref{R4}) by $a=1$ and $a^3\neq a,$ we derive $\gamma_{0}^{(i)}=0$ and $d_0^{(i)}=0,$ for $2\leq i\leq\nu.$
\end{enumerate}

\item[$\bullet$] Finally, we consider the relations (\ref{R7}) and (\ref{R8}).

\begin{enumerate}[(a)]
\item From (\ref{R7}), we obtain $a\mu_{0}+\sum\limits_{i=2}^{\nu}a_{(i)}d_{i}=0.$
In this relation, replacing $a$ with $a+1$ and  $a+\epsilon_{j} \ (2 \leq j \leq \nu),$ we get
$$(a+1)\mu_{0}+\sum_{i=2}^{\nu}a_{(i)}d_{i}=0, \quad (a+\epsilon_{j})\mu_{0}+(a_{(j)}+1)d_{j}+\sum_{i=2, i \neq j}^{\nu}a_{(i)}d_{i}=0.
$$
 Combining these relations, we obtain $\mu_{0}=0$ and $d_{i}=0$ for $2\leq i\leq\nu.$

\item    Now considering (\ref{R8}), we get 
$$
a\mu^{(i)}_{0}+\sum_{t=2}^{\nu}a_{(t)}d_{t}^{(i)}=\alpha_{0,0}a_{(i)}, \quad 2\leq i\leq\nu.
$$
 Similarly to the previous case, replacing $a$ with $a+1$ and  $a+\epsilon_{j} \ (2 \leq j \leq \nu),$ we obtain
$\mu^{(i)}_{0}=0,$  $d_{i}^{(i)}=\alpha_{0,0}$ for $2\leq i\leq\nu$
and $d_{t}^{(i)}=0$ for $t\neq i.$
 \end{enumerate}
\end{enumerate}

 As a conclusion, we obtain that $\varphi_0=\alpha_{0,0}\textrm{Id}.$
\end{proof}

 \begin{lemma}\label{thm(1)} 
$\Delta(g(G,1)) =\langle {\rm Id}, \varphi \rangle,$ where 
$\varphi(L_{k})= I_{k}$ and
$\varphi(C_{L})=  C_{LI}^{(1)}.$
\end{lemma}
\begin{proof}  
Thanks to Lemma \ref{MainLemma}, it is sufficient to consider only $\frac{1}{2}$-derivation  $\varphi_{0}$.
In the case of $\lambda =1$ and $j=0$ the restrictions \eqref{32} -- \eqref{39} have the following form:
{\footnotesize{\begin{eqnarray} 
&&2(b-a)\alpha_{0,a+b}=(b-a)(\alpha_{0,a}+\alpha_{0,b}),\label{T10}\end{eqnarray}
\begin{eqnarray}
&&2(b-a)\alpha'_{0,a+b}=(b-a)(\alpha'_{0,a}+\alpha'_{0,b}),\label{T11}\end{eqnarray}
\begin{eqnarray}
&&2d_0(a^{3}-a)-48a\gamma_{0}=(a^{3}-a)(\alpha_{0,a}+\alpha_{0,-a}),\label{T12}\end{eqnarray}
\begin{eqnarray}
&&2d_0^{(1)}(a^{3}-a)-48a\gamma^{(1)}_{0}=(a^{3}-a)(\alpha'_{0,a} + \alpha'_{0,-a}),\label{T13}\end{eqnarray}
\begin{eqnarray}
&&2(b-a)\beta_{0,a+b}=(b-a)\beta_{0,b},\label{T14}\end{eqnarray}
\begin{eqnarray}
&&2(b-a)\beta'_{0,a+b}=(b-a)(\alpha_{0,a}+\beta'_{0,b}),\label{T15}\end{eqnarray}
\begin{eqnarray}
&&2d_{1}(a^{3}-a)-48a\mu_{0}=\beta_{0,-a}(a^{3}-a),\label{T16}\end{eqnarray}
\begin{eqnarray}
&&2d^{(1)}_{1}(a^{3}-a)-48a\mu^{(1)}_{0}=(a^{3}-a)(\alpha_{0,a}+\beta'_{0,-a})\label{T17}.
\end{eqnarray}}}

We will consider the obtained relations in the following steps.

\begin{enumerate}[(a)]
    \item 
    Setting $b=0$ in (\ref{T10}) and (\ref{T11}), we have  $\alpha_{0,a}=\alpha_{0,0},$  $\alpha'_{0,a}=\alpha'_{0,0}$ for all $a\in G.$ 

      \item Taking $a=0$ in (\ref{T14}), we obtain $\beta_{0,b}=0$ for $b\neq 0.$ 
      \item    Letting $b=-a$ in (\ref{T14}), we get $\beta_{0,0}=0.$

  \item Setting $a=0$ in (\ref{T15}), we have  $\beta'_{0,b}=\alpha_{0,0}$ for $b\neq 0.$ 
  
    \item Letting  $b=-a$ in (\ref{T15}), we get $\beta'_{0,0}=\alpha_{0,0}.$ \item Taking $a=1,$ in (\ref{T12}) and (\ref{T13}), we obtain $\gamma_{0}=\gamma^{(1)}_{0}=0.$
\item Finally, putting $a=2,$ in (\ref{T12}) and (\ref{T13}), we derive that  $d_0=\alpha_{0,0}$ and $d_0^{(1)}=\alpha'_{0,0}.$ 
\item Similarly, by (\ref{T16}) and (\ref{T17}), we can get $\mu_{0}=\mu^{(1)}_{0}=0$ and $d_{1}=0,$ $d^{(1)}_{1}=\alpha_{0,0}.$

\end{enumerate}

Therefore, we obtain that $\varphi_{0}$ has the form 
$$\begin{array}{lll}
\varphi_{0}(L_{k})=\alpha L_{k}+\alpha' I_{k},&&
\varphi_{0}(I_{k})=\alpha I_{k},\\
\varphi_{0}(C_{L})=\alpha C_{L}+\alpha' C_{LI}^{(1)},&&
\varphi_{0}(C^{(1)}_{LI})=\alpha C^{(1)}_{LI}.
\end{array}$$
\end{proof}

\begin{lemma}
$\Delta(g(G,-1)) =\langle {\rm Id}, \varphi \rangle,$ where 
 $\varphi(I_{0})=   C_{L}.$
\end{lemma}

\begin{proof} The proof is similar to the proofs of Lemma \ref{thm9} and Lemma \ref{thm(1)}.
\end{proof}

\begin{theorem} 
If $(g(G, \lambda), \cdot, [\cdot, \cdot])$ is a nontrivial transposed Poisson algebra, 
then $\lambda=-1,$
$(g(G, \lambda), \cdot, [\cdot, \cdot])$ is a nontrivial   Poisson algebra, and  $I_{0} \cdot I_{0} =  \beta C_{L}.$
\end{theorem}
\begin{proof} 
Thanks to Lemmas \ref{thm9}  and \ref{thm(-2)}, any $\frac{1}{2}$-derivation on $g(G, \lambda)$ for $\lambda \neq \pm 1$ is trivial, we conclude that there are no non-trivial transposed Poisson  structures defined on $g(G, \lambda),$ for $\lambda \neq \pm 1.$
The case of $\lambda =-1$ is trivial and states the Theorem for $\lambda=-1.$

Now we aim to describe the commutative associative multiplication on transposed Poisson structures defined on $g(G, 1).$
By Lemma \ref{l1}, for each $l\in g(G, 1),$ there is a related
$\frac{1}{2}$-derivation $\varphi_{l}$ of $g(G, 1)$ such that $$\varphi_{l_{1}}(l_{2})=l_{1}\cdot l_{2}=\varphi_{l_{2}}(l_{1}),$$
with $l_{1},\ l_{2}\in g(G, 1).$ By Lemma \ref{thm(1)},  we have
$$\alpha_{L_a}C_L+\alpha'_{L_a}C^{(1)}_{LI}=\varphi_{L_a}(C_{L})=L_a\cdot C_{L}=\varphi_{C_{L}}(L_a) = \alpha_{C_L}L_a+\alpha'_{C_L}I_a.$$

Thus, we get $\alpha_{L_a}=\alpha'_{L_a}=\alpha_{C_L}=\alpha'_{C_L}=0$  for all $a\in G.$ Hence, 
$$C_L\cdot C_{L}=0,\quad L_a\cdot C_{L}=0, \quad L_a\cdot L_b=0 \quad \text{for all} \quad a,b\in G.$$

Similarly, considering $\varphi_{L_a}(C^{(1)}_{LI})=L_a\cdot C^{(1)}_{LI}=\varphi_{C^{(1)}_{LI}}(L_a),$
we obtain $L_a\cdot C^{(1)}_{LI}=C^{(1)}_{LI}\cdot C^{(1)}_{LI}=0$ for all $a\in G.$ Now, considering 
$$\alpha_{I_a}C_L+\alpha'_{I_a}C^{(1)}_{LI} = I_a\cdot C_{L}=C_{L} \cdot I_a=\varphi_{I_a}(C_{L})=\varphi_{C_{L}}(I_a) = 0, $$
we get $\alpha_{I_a}=\alpha'_{I_a}=0$ for all $a\in G,$ which implies $$I_a\cdot I_b=0, \quad I_a\cdot C^{(1)}_{LI}=0, \quad I_a\cdot L_b=0 \quad \text{for all} \quad  a, b \in G.$$ 
Hence, all commutative associative multiplications on $g(G, 1)$ are trivial. 
\end{proof}

\section{Transposed Poisson  structures on not-finitely graded Lie algebra $\widehat{W}(G)$}

Let $G$ be any nontrivial additive subgroup of $\mathbb{C}$, and $\mathbb{C}[G\times  \mathbb{Z}_{+}]$
the semigroup algebra of $G\times  \mathbb{Z}_{+}$ with basis $\{x^{\alpha}t^i \ | \ \alpha \in G, i \in \mathbb{Z}_{+}\}$ and product $(x^{\alpha}t^i) \cdot (x^{\beta}t^j) = x^{\alpha+\beta}t^{i+j}.$ Let $\partial_x,$ $\partial_t$
be the derivations of  $\mathbb{C}[G\times  \mathbb{Z}_{+}]$ defined by 
$\partial_x(x^{\alpha}t^i) = \alpha x^{\alpha-1}t^i,$ $\partial_t(x^{\alpha}t^i) = i x^{\alpha}t^{i-1}$
  for  $\alpha \in G, i \in \mathbb{Z}_{+}.$ Denote 
 $\partial = \partial_x + t^2 \partial_t.$ Then the Lie algebra
$W(G)$ is $\mathbb{C}[G\times \mathbb{Z}_{+}]\partial$
 with basis $\{L_{\alpha, i} = x^{\alpha}t^i \partial \ | \ \alpha \in G, i \in \mathbb{Z}_{+}\}$ and relation 
 \begin{equation*}
     [L_{\alpha, i}, L_{\beta, j}] = (\beta - \alpha)L_{\alpha+\beta, i+j} + (j-i) L_{\alpha+\beta, i+j+1}.
 \end{equation*}


The Lie algebra $W(G)$ has the unique universal central extension $\widehat{W}(G)$
with one dimensional center $ \langle C \rangle$ and relation 
\begin{equation}\label{cab}
  [L_{\alpha, i}, L_{\beta, j}] = (\beta - \alpha)L_{\alpha+\beta, i+j} + (j-i) L_{\alpha+\beta, i+j+1}+\delta_{\alpha,-\beta}\delta_{i,-j}\frac{\alpha^3-\alpha}{12}C.  
\end{equation}
$\widehat{W}(G)$ contains the Virasoro algebra
$Vir(G)=\langle L_{\alpha, 0,},C \ | \ \alpha\in G\rangle.$
We refer to $\widehat{W}(G)$ as a not-finitely graded  Heisenberg-Virasoro algebra.
The Lie algebra  $\widehat{W}(G)$ is a $G$-graded algebra.
Namely, $\widehat{W}(G)=\bigoplus_{\alpha\in G}\widehat{W}(G)_{\alpha},$ where 
  $  \widehat{W}(G)_{0}= \langle  L_{0,i}, C \ | \ i\in \mathbb{Z}_{+}\rangle$ and $ \widehat{W}(G)_{\alpha \in G\setminus\{0\}}= \langle  L_{\alpha,i} \ | \ i\in \mathbb{Z}_{+} \rangle.$

\begin{theorem} \label{main4}
$\Delta(\widehat{W}(G)) =\langle {\rm Id}, \varphi \rangle,$ where 
$\varphi(L_{\alpha,i})=\sum\limits_{d\in G }\sum\limits_{k\geq 1}a^{d,k}L_{\alpha+d,i+k}.$
\end{theorem}
\begin{proof}

By Lemma \ref{l01}, we get 
$\Delta(\widehat{W}(G))=\bigoplus\limits_{\alpha\in G}\Delta_{\alpha}(\widehat{W}(G)).$
To determine $\Delta(\widehat{W}(G)),$ we need to calculate 
every homogeneous subspace $\Delta_{\alpha}(\widehat{W}(G)).$
Let $\varphi_d \in  \Delta_d(\widehat{W}(G)).$  
Since $\operatorname{Ann}(\widehat{W}(G))$ is invariant under any $\frac{1}{2}$-derivation, we can assume that 
$$\varphi_d(L_{\alpha,i})=\sum_{k\in\mathbb{Z}_{+}}a^{d,k}_{\alpha,i}L_{\alpha+d,k}+\delta_{\alpha+d,0}\delta_{i,0}a_{\alpha}C, \quad
\varphi_d(C)=\delta_{d,0}\lambda C.$$
To determine these coefficients, we start with applying $\varphi_d$
to both side of (\ref{cab}) 
and obtain
{\small\begin{equation}\label{maini}
\begin{array}{cl}
2\sum\limits_{k\in\mathbb{Z}} & \left((\beta-\alpha)a_{\alpha+\beta,i+j}^{d,k}+(j-i)a_{\alpha+\beta,i+j+1}^{d,k}\right)  L_{\alpha+\beta+d,k}=\\[2mm]
&\multicolumn{1}{r}{=\sum\limits_{k\in\mathbb{Z}}a_{\alpha,i}^{d,k}\left((\beta-\alpha-d)L_{\alpha+\beta+d,j+k}+(j-k)L_{\alpha+\beta+d,j+k+1}\right)+}\\[2mm]
&\multicolumn{1}{r}{+\sum\limits_{k\in\mathbb{Z}}a_{\beta,j}^{d,k}\left((\beta-\alpha+d)L_{\alpha+\beta+d,i+k}+(k-i)L_{\alpha+\beta+d,i+k+1}\right).}
\end{array}
\end{equation}
\begin{equation}\label{mainid6}
\begin{array}{ccll}
&2(\beta-\alpha)\delta_{\alpha+\beta,-d}\delta_{i,-j}a_{\alpha+\beta}C+2(j-i)\delta_{\alpha+\beta,-d}\delta_{i+j,-1}a_{\alpha+\beta}C+
2\delta_{\alpha,-\beta}\delta_{i,-j}\delta_{d,0}\frac{\alpha^3-\alpha}{12}\lambda C\\[2mm]
&=\sum\limits_{k\in\mathbb{Z}_{+}}a^{d,k}_{\alpha,i}\delta_{\alpha+\beta,-d}\delta_{j,-k}\frac{(\alpha+d)^3-(\alpha+d)}{12}C+\sum\limits_{k\in\mathbb{Z}_{+}}a^{d,k}_{\beta,j}\delta_{\alpha+\beta,-d}\delta_{i,-k}\frac{\alpha^3-\alpha}{12}C.
\end{array}
\end{equation}}

The relation (\ref{maini}) corresponds to the case $n=1$ in the proof of  Theorem 12    in \cite{kkz}. It follows that 
$a_{\alpha,i}^{d,k}=a_{0,0}^{d,k-i},$ $k\geq i.$
We simply denote $a^{d,k}=a_{0,0}^{d,k}.$ Setting $i=j=0$ in (\ref{mainid6}), we deduce the relation
\begin{equation}\label{mainid7}
 24(\beta-\alpha)\delta_{\alpha+\beta,-d}a_{\alpha+\beta}
+2\delta_{\alpha,-\beta}\delta_{d,0}{(\alpha^3-\alpha)}\lambda =
 a^{d,0}\delta_{\alpha+\beta,-d}\left({(\alpha+d)^3-(\alpha+d)}+
{\alpha^3-\alpha}\right).
\end{equation}

For $d\neq 0$ putting $\alpha=-d$ and $\beta=0$ in (\ref{mainid7}), we have 
\begin{center}
$-24(2\alpha+d)a_{-d} =
 a^{d,0}\left((\alpha+d)^3-(\alpha+d)+
\alpha^3-\alpha\right).$
\end{center}

This equality, due to arbitrary $\alpha$, gives that
$a^{d,0}=a_{-d}=0$ for all $d\in G\setminus\{0\}.$

For $d=0$ setting $\beta=-\alpha$ in (\ref{mainid7}), we obtain  the relation
$
  24\alpha a_0=( \lambda - a^{0,0} )(\alpha^2-\alpha).   
$
Again, due to arbitrary $\alpha$, we get $a_0=0$ and $\lambda=a^{0,0}.$
Hence it proves that $\varphi_d$ has the form 
$$\begin{array}{rclrcl} 
\varphi_{d\neq 0}(L_{\alpha,i})&=&\sum\limits_{k\geq i+1}a^{d,k-i}L_{\alpha+d,k}, \\
\varphi_0(L_{\alpha,i})&=&\sum\limits_{ k\geq i}a^{0,k-i}L_{\alpha,k}, & \varphi_0(C)&=&a^{0,0}C.
\end{array}$$

\end{proof}

\begin{theorem} Let $(\widehat{W}(G),\cdot,[\cdot,\cdot])$ be a transposed Poisson algebra structure defined on the Lie algebra $\widehat{W}(G).$ Then $(\widehat{W}(G),\cdot,[\cdot,\cdot])$ is not Poisson algebra and 
there is a finite set of elements $\{ a^{d,k}| d \in G, k \in \mathbb{Z}_+\},$ such that  $$L_{\alpha,i}\cdot L_{\beta,j}=\sum_{d\in G}\sum_{k\in\mathbb{Z}_{+}}a ^{d,k}L_{d+\alpha+\beta,k+i+j+1}.$$
 
\end{theorem}
\begin{proof} We aim describe the multiplication $\cdot.$ By Lemma \ref{l1}, for every element  $l\in \widehat{W}(G),$ there is a related
$\frac{1}{2}$-derivation $\varphi_{l}$ of $\widehat{W}(G)$ such that $\varphi_{l_{1}}(l_{2})=l_{1}\cdot l_{2}=\varphi_{l_{2}}(l_{1}),$
with $l_{1}, l_{2}\in \widehat{W}(G).$ By Theorem \ref{main4}, we have the description of all $\frac 12$-derivations of $\widehat{W}(G)$.

Now we consider  the multiplication $C\cdot L_{\alpha,i}.$ Then 
\begin{longtable}{rcl}
$\varphi_{L_{\alpha,i}}(C)=L_{\alpha,i} \cdot C $&$= $&$C\cdot L_{\alpha,i}=\varphi_{C}(L_{\alpha,i}),$\\
$a_{L_{\alpha,i}}^{0,0}C$&$
=$&$\sum\limits_{d\in G}\sum\limits_{k\geq 1}a_C^{d,k}L_{\alpha+d,k+i}+a_C^{0,0}L_{\alpha,k}.$
\end{longtable}

Thus, we get $a_{L_{\alpha,i}}^{0,0}=0$ and $a_C^{d,k}=0,$
i.e., $C\cdot L_{\alpha,i}=0.$

Then we determine  the multiplication $L_{0,0}\cdot L_{\alpha,i}$ for all $i\in \mathbb{Z}_{+}$ and obtain
$$\varphi_{L_{0,0}}(L_{\alpha,i})=L_{0,0}\cdot L_{\alpha,i}=L_{\alpha,i} \cdot L_{0,0}=\varphi_{L_{\alpha,i}}(L_{0,0}),$$
$$\sum_{d\in G}\sum_{k\geq 1}a_{L_{0,0}}^{d,k}L_{\alpha+d,k+i}
=\sum_{d\in G}\sum_{k\geq 1}a_{L_{\alpha,i}}^{d,k}L_{d,k}.$$

Hence, for all $\alpha, d\in G,$ we get
\begin{equation}\label{p4}
a_{L_{\alpha,i}}^{d,k}=0, \quad 1\leq k\leq i; \quad
a_{L_{\alpha,i}}^{d,k}=a_{L_{0,0}}^{d-\alpha,k-i}, \quad  k\geq 
i+1.
\end{equation}

Next we determine the multiplication $L_{\alpha,i}\cdot L_{\beta,j}$ for all $\alpha,\beta\in G,$ $i,j\in \mathbb{Z}_{+}$ and obtain
$$L_{\alpha,i}\cdot L_{\beta,j}= \varphi_{L_{\alpha,i}}(L_{\beta,j})=
\sum_{d\in G}\sum_{k\geq 1}a_{L_{\alpha,i}}^{d,k}L_{\beta+d,j+k}
=\sum_{d\in G}\sum_{k\geq i+1}a_{L_{0,0}}^{d-\alpha,k-i}L_{\beta+d,j+k}.$$

Hence, by (\ref{p4}) we get 
 $$L_{\alpha,i}\cdot L_{\beta,j}=\sum_{d\in G}\sum_{k\in\mathbb{Z}_{+}}a_{L_{0,0}}^{d,k}L_{d+\alpha+\beta,k+i+j+1}.$$

It is easy to see that $\cdot$  is associative.
The product $\cdot$
gives a transposed Poisson algebra structure defined on the Lie algebra $\widehat{W}(G).$  It is easy to see that $\widehat{W}(G)\cdot [\widehat{W}(G),\widehat{W}(G)]\neq 0$ and  $(\widehat{W}(G),\cdot,[\cdot,\cdot])$ is a non-Poisson algebra.
It gives the complete statement of the theorem.
\end{proof}
 
\section{Transposed Poisson structures on not-finitely graded Lie algebra $\widetilde{W}(G)$}

In this section, we consider a central extension
$\widetilde{W}(G)$ of the algebra $W(G)$ which was  studied in \cite{su}. The algebra
$W(G)$ is the Lie algebra
 with basis $\{L_{\alpha, i}  \ | \ \alpha \in G, i \in \mathbb{Z}\}$ and the multiplication table
 \begin{equation*}
     [L_{\alpha, i}, L_{\beta, j}] = (\beta-\alpha)L_{\alpha+\beta, i+j} + (j-i) L_{\alpha+\beta, i+j-1}.
 \end{equation*}
 
 $\widetilde{W}(G)$ is the Lie algebra with basis  $\{L_{\alpha, i} , C \ | \ \alpha \in G, i \in \mathbb{Z}\}$ and the following multiplication table:
\begin{equation}\label{ab2}
\begin{array}{lcl}
      [L_{\alpha, i}, L_{\beta, j}] &= &(\beta-\alpha)L_{\alpha+\beta, i+j} + (j-i) L_{\alpha+\beta, i+j-1} +\\[2mm]
  &   & +\delta_{\alpha,-\beta}(\delta_{i+j,-1}\alpha^3+3i\delta_{i+j,0}\alpha^2 +3i(i-1)\delta_{i+j,1}\alpha+i(i-1)(i-2)\delta_{i+j,2})C.
\end{array}
 \end{equation}


\noindent $\widetilde{W}(G)$ is $G$-graded:
$\widetilde{W}(G)=\bigoplus\limits_{\alpha\in G}\widetilde{W}(G)_{\alpha},$ where
$\widetilde{W}(G)_{0}=\langle L_{0,i}, \ C \rangle_{i\in \mathbb{Z}}$ 
and $\widetilde{W}(G)_{\alpha\neq 0}=\langle L_{\alpha,i} \rangle_{i\in \mathbb{Z}}.$

\begin{theorem} \label{main5}
  $\Delta(\widetilde{W}(G)) =\langle {\rm Id } \rangle.$  \end{theorem}
\begin{proof}
 By Lemma \ref{l01}, we have 
$\Delta(\widetilde{W}(G))=\bigoplus\limits_{\alpha\in G}\Delta_{\alpha}(\widetilde{W}(G)).$
To determine $\Delta(\widetilde{W}(G)),$ we need to calculate 
every homogeneous subspace $\Delta_{\alpha}(\widetilde{W}(G)).$
Let $\varphi_d \in \Delta_{d}(\widetilde{W}(G)).$ 
It is known that  $\operatorname{Ann}(\widetilde{W}(G))$ is invariant under any $\frac{1}{2}$-derivation. So  we can assume that 
\begin{center}$\varphi_d(L_{\alpha,i})=\sum\limits_{k\in\mathbb{Z}}a^{d,k}_{\alpha,i}L_{\alpha+d,k}+\delta_{\alpha+d,0}a_{\alpha,i}C, \quad \varphi_d(C)=\delta_{d,0}\lambda C.$
\end{center}

To determine these coefficients, we start with applying $\varphi_d$
to both side of (\ref{ab2}) and obtain
{\small \begin{equation}\label{main}
\begin{array}{rcll}
2\sum\limits_{k\in\mathbb{Z}}\left((\beta-\alpha)a_{\alpha+\beta,i+j}^{d,k}+(j-i)a_{\alpha+\beta,i+j-1}^{d,k}\right)L_{\alpha+\beta+d,k}&=&\\
\sum\limits_{k\in\mathbb{Z}}a_{\alpha,i}^{d,k}\big((\beta-\alpha-d)L_{\alpha+\beta+d,j+k}&+&(j-k)L_{\alpha+\beta+d,j+k-1}\big)+\\ 
\sum\limits_{k\in\mathbb{Z}}a_{\beta,j}^{d,k}\big((\beta-\alpha+d)L_{\alpha+\beta+d,i+k}&+&(k-i)L_{\alpha+\beta+d,i+k-1}\big).
\end{array}
\end{equation}
\begin{equation}\label{maini6}
\begin{array}{rcll}
\multicolumn{1}{l}{2\delta_{\alpha+\beta,-d}\big((\beta-\alpha)a_{\alpha+\beta,i+j}
+(j-i)a_{\alpha+\beta,i+j-1}\big) \ +}\\
\multicolumn{1}{r}{+ 2\delta_{\alpha,-\beta}\delta_{d,0}\big(\delta_{i+j,-1}\alpha^3+3i\delta_{i+j,0}\alpha^2+3i(i-1)\delta_{i+j,1}\alpha+i(i-1)(i-2)\delta_{i+j,2}\big)\lambda=}\\[2mm]
\multicolumn{1}{l}{\delta_{\alpha+\beta,-d} \Big(\sum\limits_{k\in\mathbb{Z}}a_{\alpha,i}^{d,k}\big(\delta_{j+k,-1}(\alpha+d)^3+3k\delta_{j+k,0}(\alpha+d)^2+3k(k-1)\delta_{j+k,1}(\alpha+d)+k(k-1)(k-2)\delta_{j+k,2}\big)+}\\
\multicolumn{1}{r}{+\sum\limits_{k\in\mathbb{Z}}a^{d,k}_{\beta,j}\big(\delta_{i+k,-1}\alpha^3+3i\delta_{i+k,0}\alpha^2
+3i(i-1)\delta_{i+k,1}\alpha+i(i-1)(i-2)\delta_{i+k,2}\big)\Big).}
\end{array}
\end{equation}}

\noindent Similarly to the previous case by (\ref{main}), we can get 
$a_{\alpha,i}^{d,k}=a_{0,0}^{d,k-i}$ for $k,i\in\mathbb{Z}. $
For $d\neq 0$ putting $\alpha=-d$, $\beta =i=j=0$  in (\ref{maini6}), we have 
$a_{-d,0} =-\frac{d^2}{2}a^{d,-1}.$
Then
setting $\beta=-\alpha-d$ in (\ref{maini6}) with $i=j=0$, we get $\alpha(d+\alpha)(d+2\alpha)a^{d,-1}=0.$  Choosing $\alpha\notin \{-\frac{d}{2};-d;0\},$ we obtain
$a^{d,-1}=0.$ Hence, $a_{-d,0}=0$ for all $d\in G\setminus\{0\}.$

\noindent Next putting  $\beta=-\alpha-d,$ $i=0$ (\ref{maini6}) for 
$d\neq 0$ it gives
\begin{equation*} 
  -2(2\alpha+d) a_{-d,j}+2ja_{-d,j-1}=
\end{equation*}
  \begin{equation} ((\alpha+d)^3 +\alpha^3)a^{d,-j-1}-3j(\alpha+d)^2 a^{d,-j}+3j(j-1)(\alpha+d)a^{d,1-j} -j(j-1)(j-2)a^{d,2-j}.\label{p9}
\end{equation}
For $d\neq 0$ letting $\alpha=-d,$ $\beta=j=0,$ and 
$\alpha=-d,$ $\beta=i=0,$  in (\ref{maini6}), we can get
\begin{equation}\label{p10}
  2d a_{-d,i}-2ia_{-d,i-1}=-d^3 a^{d,-i-1}+3id^2 a^{d,-i}-3i(i-1)da^{d,1-i} +i(i-1)(i-2)a^{d,2-i}
\end{equation}
\begin{equation}\label{p11}
  2d a_{-d,j}+2ja_{-d,j-1}=-d^3 a^{d,-j-1}-j(j-1)(j-2)a^{d,2-j}
\end{equation}
Combining (\ref{p9}) and (\ref{p10}) it gives
\begin{equation}\label{p12}
4 a_{-d,i}=-(2\alpha^2+3\alpha d+3d^2) a^{d,-i-1}+3i(\alpha+2d) a^{d,-i}-3i(i-1)a^{d,1-i}
\end{equation}
for $\alpha\neq0.$ Similarly, from (\ref{p10}) and (\ref{p11}), we obtain
\begin{equation}\label{p13}
4 a_{-d,i}=-2d^2 a^{d,-i-1}+3id a^{d,-i}-3i(i-1)a^{d,1-i}.
\end{equation}
Then subtracting (\ref{p12}) from (\ref{p13}), we deduce this equality
\begin{equation}\label{p14}
(\alpha+d)((2\alpha+d)a^{d,-i-1}-3ia^{d,-i})=0.
\end{equation}
This equality, due to arbitrary $\alpha$, gives that
$a^{d,i}=0$ for all $i\in \mathbb{Z}\setminus\{0\}.$ Putting 
$i=-1$ in (\ref{p14}), we have $a^{d,0}=0$ and from (\ref{p13}), we deduce 
$a_{-d,i}=0$ for all $i\in \mathbb{Z}.$ 

For $d=0$ setting $\beta=-\alpha,$  $i=j=0$ in (\ref{maini6}), we obtain  the relation
$2a_{0,0}=-\alpha^2a^{0,-1}. $
Due to arbitrary $\alpha$, we get $a^{0,-1}=a_{0,0}=0.$
For $d=0$ and $j\neq 0$ putting $\alpha=\beta=i=0$ in (\ref{maini6}), we can get 
\begin{equation}\label{p5}
 2a_{0,j-1}=-(j-1)(j-2)a^{0,2-j}.   
\end{equation}
Setting $\beta=-\alpha,$ $i=0$ in (\ref{maini6}) with (\ref{p5}), we obtain
\begin{equation}\label{p6}
 -4\alpha a_{0,j}+2\delta_{j,-1}\alpha^3\lambda =2\alpha^3a^{0,-j-1}-3j\alpha^2a^{0,-j}+3j(j-1)\alpha a^{0,1-j}.   
\end{equation}
Now we consider the case of $j=-1$ in (\ref{p6}), then we deduce 
\begin{equation}\label{p7}
 2\alpha^2(\lambda-a^{0,0}) =4a_{0,-1}+3\alpha a^{0,1}+6a^{0,2}   
\end{equation}
for $\alpha\neq 0.$ Replacing $\alpha$ with $-\alpha$ in (\ref{p7}), then combining them we have  
$$a^{0,1}=0,\quad \lambda=a^{0,0}, \quad a_{0,-1}=-\frac{3}{2}a^{0,2}.$$

\noindent Next we consider the case of $j\neq -1$ in (\ref{p6}) with together (\ref{p5}) it gives 
\begin{equation}\label{p8}
 j(j-1)a^{0,1-j} =3j\alpha a^{0,-j}-2\alpha^2 a^{0,-j-1}.   
\end{equation}
Similarly, as described above, replacing $\alpha$ with $-\alpha$ in (\ref{p8}) and subtracting them we obtain 
$a^{0,j}=0$ for $j\neq 0.$ So from (\ref{p5}), we deduce $a_{0,j}=0$ for all $j\in \mathbb{Z}.$
As a conclusion, we obtain $\varphi_0=\alpha^{0,0}\textrm{Id}.$
\end{proof}

Thanks to Lemma \ref{lempr}, we have the following statement.

\begin{theorem} There are no non-trivial transposed Poisson  structures defined on $\widetilde{W}(G).$
\end{theorem}

\section{Transposed Poisson structures on 
 not-finitely graded   algebra $\widetilde{HW}(G)$
} 

In this section, we consider one-dimensional central extension  $\widetilde{HW}(G)$  of  \textit{generalized Heisenberg-Virasoro algebra}  $HW(G,-1)$. Let $G$ be any nontrivial additive subgroup of $\mathbb{C}$, and $\mathbb{C}[G\times  \mathbb{Z}_{+}]$
the semigroup algebra of $G\times  \mathbb{Z}_{+}$ with basis $\{x^{\alpha,i}=x^{\alpha}t^i \ | \ \alpha \in G, i \in \mathbb{Z}_{+}\}$ and product $x^{\alpha,i} \cdot x^{\beta,j} = x^{\alpha+\beta,i+j}.$
Let $\partial_x,$ $\partial_t$
be the derivations of  $\mathbb{C}[G\times  \mathbb{Z}_{+}]$ defined by 
$\partial_x(x^{\alpha,i}) = \alpha x^{\alpha-1,i},$ $\partial_t(x^{\alpha,i}) = i x^{\alpha,i-1}$
    for $\alpha \in G, i \in \mathbb{Z}_{+}.$ Denote 
 $\partial = \partial_x + \partial_t.$ Then $HW(G,-1)$ the Lie algebra
with  the 
  basis $\{L_{\alpha, i} = x^{\alpha,i} \partial, \    H_{\beta, j} = x^{\beta,j} \ | \ \alpha,\beta \in G, i,j \in \mathbb{Z}_{+}\}$ and the following multiplication table: 
 \begin{longtable}{lcl}
 $[L_{\alpha, i}, L_{\beta, j}] $&$=$&$ (\beta - \alpha)L_{\alpha+\beta, i+j} + (j-i) L_{\alpha+\beta, i+j-1},$ \\[2mm]
 $[L_{\alpha, i}, H_{\beta, j}] $&$=$&$ \beta H_{\alpha+\beta, i+j} + j H_{\alpha+\beta, i+j-1}.$ \\
 \end{longtable}

\begin{theorem}(see, \cite{kkz}) 
Let $\varphi$ be a $\frac{1}{2}$-derivation of the algebra $HW,$ then 
\begin{longtable}{lcl}
$\varphi(L_{\alpha,i})$&$=$&$\sum\limits_{d\in G\setminus\{0\}}\sum\limits_{k\in\mathbb{Z}_{+}}\frac{(-d)^k}{k!}a^{d,0}L_{\alpha+d,k}+a^{0,0}L_{\alpha,i},$\\
$\varphi(H_{\alpha,i})$&$=$&$\sum\limits_{d\in G\setminus\{0\}}\sum\limits_{k\in\mathbb{Z}_{+}}\frac{(-d)^k}{k!}a^{d,0}H_{\alpha+d,k}+a^{0,0}H_{\alpha,i}.$

\end{longtable}

\end{theorem}

The algebra $\widetilde{HW}(G),$ that was  obtained in \cite{fzy15},
 has the following multiplication table: 
\begin{eqnarray}
     &&[L_{\alpha, i}, L_{\beta, j}] = (\beta - \alpha)L_{\alpha+\beta, i+j} + (j-i) L_{\alpha+\beta, i+j-1}+\delta_{\alpha,-\beta}\delta_{i,-j}\frac{\alpha^3-\alpha}{12}C_L,\label{ad1}\\
 &&[L_{\alpha, i}, H_{\beta, j}] = \beta H_{\alpha+\beta, i+j} + j H_{\alpha+\beta, i+j-1}+\delta_{\alpha,-\beta}\delta_{i,-j}(\alpha^2-\alpha)C_{LH},\label{cd1}\\
 &&[H_{\alpha, i}, H_{\beta, j}] = \delta_{\alpha,-\beta}\delta_{i,-j}\alpha C_H.\label{bd1}
 \end{eqnarray}

\begin{center}
$\widetilde{HW}(G)=\bigoplus_{\alpha\in G}\widetilde{HW}(G)_{\alpha},$ where
$\widetilde{HW}(G)_{0}=\langle L_{0,i}, C_L, C_{LH},C_H  \rangle_{i\in \mathbb{Z}_{+}}$  and 
$\widetilde{HW}(G)_{\alpha\neq 0}=\langle L_{\alpha,i} \rangle_{i\in \mathbb{Z}_{+}}.$
\end{center}

\begin{theorem} \label{main8}
 $\Delta(\widetilde{HW}(G))=\langle {\rm Id}\rangle.$  
 \end{theorem}
\begin{proof}
By Lemma \ref{l01}, we get  $\Delta(\widetilde{HW}(G))=\bigoplus\limits_{\alpha\in G}\Delta_{\alpha}(\widetilde{HW}(G)).$
To determine $\Delta(\widetilde{HW}(G)),$ we need to calculate 
every homogeneous subspace $\Delta_{\alpha}(\widetilde{HW}(G)).$
Let $\varphi_d \in \Delta_d(\widetilde{HW}(G)).$ 
It is known that $\operatorname{Ann}(\widetilde{HW}(G))$ is invariant under any $\frac{1}{2}$-derivation we can assume that 

\begin{longtable}{rcl}
$\varphi_d(L_{\alpha,i})$&$=$&$\sum\limits_{k\in\mathbb{Z}_+}a^{d,k}_{\alpha,i}L_{\alpha+d,k}+\sum\limits_{k\in\mathbb{Z}}b_{\alpha,i}^{d,k}H_{\alpha+d,k}+\delta_{\alpha,-d}\left(\gamma_{\alpha,i}^{1,1}C_L+\gamma_{\alpha,i}^{1,2}C_{LH}+\gamma_{\alpha,i}^{1,3}C_H\right),$\\
$\varphi_d(H_{\alpha,i})$&$=$&$\sum\limits_{k\in\mathbb{Z}}f_{\alpha,i}^{d,k}L_{\alpha+d,k}+\sum\limits_{k\in\mathbb{Z}}c_{\alpha,i}^{d,k}H_{\alpha+d,k}+\delta_{\alpha,-d}\left(\gamma_{\alpha,i}^{2,1}C_L+\gamma_{\alpha,i}^{2,2}C_{LH}+\gamma_{\alpha,i}^{2,3}C_H\right),$\\
$\varphi_d(C_L)$&$=$&$\delta_{d,0}\left(\gamma_L^1C_L+\gamma_L^{2}C_{LH}+\gamma_{L}^{3}C_H\right),$\\
$\varphi_d(C_{LH})$&$=$&$\delta_{d,0}\left(\gamma_{LH}^1C_L+\gamma_{LH}^{2}C_{LH}+\gamma_{LH}^{3}C_H\right),$\\
 $\varphi_d(C_H)$&$=$&$\delta_{d,0}\left(\gamma_H^1C_L+\gamma_H^{2}C_{LH}+\gamma_{H}^{3}C_H\right).$
\end{longtable}

\noindent From  \cite[Theorem 15]{kkz}, we can get 
\begin{center}$b_{\alpha,i}^{d,k}=f_{\alpha,i}^{d,k}=0,$ \ for \ $\alpha\in G, \ i\in \mathbb{Z};$ \quad
$a_{\alpha,i}^{d,k}=c_{\alpha,i}^{d,k}=\frac{(-d)^k}{k!}a_{0,0}^{d,0},$ \  for \  $ d\neq 0.$
\end{center}
We simply denote $a_{0,0}^{d,0}=a^{d,0}.$
To determine other coefficients, we start with applying $\varphi_d$
to both side of (\ref{bd1}) and obtain
{\small \begin{equation}\label{id1}
2\delta_{\alpha,-\beta}\delta_{i,-j}\delta_{d,0}\alpha\left(\gamma_H^1C_L+\gamma_H^{2}C_{LH}+\gamma_{H}^{3}C_H\right)=\sum_{k\in\mathbb{Z}_+}\frac{(-d)^k}{k!}\delta_{\alpha+\beta,-d}a^{d,0}\left(\delta_{j,-k}(\alpha+d)+\delta_{i,-k}\alpha\right)C_{H},  \end{equation}}
for $d\neq0.$
Taking $\beta=-\alpha-d,$ $i=k=0,$ $j\neq0$ in (\ref{id1}), it gives $a^{d,0}=0.$ For $d=0$,
we have $$2\delta_{\alpha,-\beta}\delta_{i,-j}\alpha\left(\gamma_H^1C_L+\gamma_H^{2}C_{LH}+\gamma_{H}^{3}C_H\right)=2a^{0,0}\delta_{\alpha,-\beta}\delta_{i,-j}\alpha C_H.$$
In this case putting $\beta=-\alpha,$ $i=j=0$, we can get $\gamma_H^1=\gamma_H^2=0$ and  $\gamma_H^3=a^{0,0}.$

Next we apply $\varphi_d$ to both side of (\ref{cd1}) and this implies 

{\small \begin{equation}\label{id2}
\begin{array}{ccll}
 \beta\left(\gamma_{\alpha+\beta,i+j}^{2,1}C_L+\gamma_{\alpha+\beta,i+j}^{2,2}C_{LH}+\gamma_{\alpha+\beta,i+j}^{2,3}C_H\right) 
=- j\left(\gamma_{\alpha+\beta,i+j-1}^{2,1}C_L+\gamma_{\alpha+\beta,i+j-1}^{2,2}C_{LH}+\gamma_{\alpha+\beta,i+j-1}^{2,3}C_H\right),
\end{array}
\end{equation}}

\noindent for $\alpha+\beta+d=0$ and   $d\neq 0.$ 

Putting $\alpha=j=0$, $\beta =-d$  in (\ref{id2}), we have 
$\gamma_{-d,i}^{2,1}=\gamma_{-d,i}^{2,2}=\gamma_{-d,i}^{2,3}=0$ for all $i\in\mathbb{Z}_{+}.$

For $d=0$ we have 
{\small \begin{equation}\label{id3}
\begin{array}{rlllcl}
\delta_{\alpha,-\beta} \Big( & \beta\big(\gamma_{\alpha+\beta,i+j}^{2,1}C_L+\gamma_{\alpha+\beta,i+j}^{2,2}C_{LH}+\gamma_{\alpha+\beta,i+j}^{2,3}C_H\big)\\
&+j\big(\gamma_{\alpha+\beta,i+j-1}^{2,1}C_L+\gamma_{\alpha+\beta,i+j-1}^{2,2}C_{LH}+\gamma_{\alpha+\beta,i+j-1}^{2,3}C_H\big)\\
&\multicolumn{1}{r}{+\delta_{i,-j}(\alpha^2-\alpha)\big(\gamma_{LH}^1C_L+\gamma_{LH}^{2}C_{LH}+\gamma_{LH}^{3}C_H\big) \Big)} &=&\delta_{\alpha+\beta,0}\delta_{i+j,0}(\alpha^2-\alpha)a^{0,0}C_{LH}
\end{array}
\end{equation}}
Then
setting $\beta=-\alpha$  in (\ref{id3}) with $i=j=0$, we get the following relation 
$$-\alpha\left(\gamma_{0,0}^{2,1}C_L+\gamma_{0,0}^{2,2}C_{LH}+\gamma_{0,0}^{2,3}C_H\right)
+(\alpha^2-\alpha)\left(\gamma_{LH}^1C_L+\gamma_{LH}^{2}C_{LH}+\gamma_{LH}^{3}C_H\right)=(\alpha^2-\alpha)a^{0,0}C_{LH}.$$
In this relation replacing $\alpha$ with $-\alpha$, then combining them, we deduce 
$$\gamma_{0,0}^{2,1}=\gamma_{0,0}^{2,2}=\gamma_{0,0}^{2,3}=\gamma_{LH}^1=\gamma_{LH}^3=0, \quad \gamma_{LH}^2=a^{0,0}.$$

Finally, applying $\varphi_d$ to both side of (\ref{ad1}), we derive
{\small \begin{equation}\label{id4}
\begin{array}{ccll}
\multicolumn{1}{l}{(\beta-\alpha)\left(\gamma_{\alpha+\beta,i+j}^{1,1}C_L+\gamma_{\alpha+\beta,i+j}^{1,2}C_{LH}+\gamma_{\alpha+\beta,i+j}^{1,3}C_H\right) +}\\
(j-i)\left(\gamma_{\alpha+\beta,i+j-1}^{1,1}C_L+\gamma_{\alpha+\beta,i+j-1}^{1,2}C_{LH}+\gamma_{\alpha+\beta,i+j-1}^{1,3}C_H\right)=0,
\end{array}
\end{equation}}

\noindent for $\alpha+\beta+d=0$ and $d\neq 0.$ Letting $\beta=\alpha,$ $d=-2\alpha,$ $j=1$ in (\ref{id4}), we can get 
$\gamma_{-d,i}^{1,1}=\gamma_{-d,i}^{1,2}=\gamma_{-d,i}^{1,3}=0$ for all $i\in\mathbb{Z}_{+}.$
Similarly to the previous case, we deduce
$$\gamma_{0,0}^{1,1}=\gamma_{0,0}^{1,2}=\gamma_{0,0}^{1,3}=\gamma_{L}^2=\gamma_{L}^3=0, \quad \gamma_{L}^1=a^{0,0}.$$

As a conclusion, we obtain $\varphi_0=\alpha^{0,0}\textrm{Id}.$
\end{proof}

Thanks to Lemma \ref{lempr}, we have the following statement.

\begin{theorem} There are no non-trivial transposed Poisson structures defined on  $\widetilde{HW}(G).$
\end{theorem}



\end{document}